\documentclass[sn-mathphys-num]{sn-jnl}% Math and Physical Sciences Numbered Reference Style 
\usepackage{graphicx}
\usepackage{amsmath,amssymb,amsthm}

\usepackage{lmodern}
\theoremstyle{plain}

\newtheorem{theorem}{Theorem}
\newtheorem{proposition}{Proposition}
\newtheorem{lemma}{Lemma}
\newtheorem{corollary}{Corollary}

\newcommand{\T}{\mathbf T}
\newcommand{\Tss}{\mathbf{T}_{ss}}

\newcommand{\Xt}{\mathbf X_t}
\newcommand{\X}{\mathbf X}
\newcommand{\Xs}{\mathbf X_s}
\newcommand{\Xss}{\mathbf X_{ss}}
\newcommand{\Tt}{\mathbf T_t}

\newcommand{\N}{\mathbb N}
\newcommand{\Z}{{\mathbb Z}}
\newcommand{\II}{{\mathcal I}}

\begin{document}
	
	\title[]{Regular Polygonal Vortex Filament Evolution and Exponential Sums}
	
	
	\author[1]{\fnm{Fernando} \sur{Chamizo}}\email{fernando.chamizo@uam.es}
	
	\author*[2]{\fnm{Francisco} \sur{de la Hoz}}\email{francisco.delahoz@ehu.eus}
	
	\affil[1]{\orgdiv{Departamento de Matem\'aticas and ICMAT}, \orgname{Universidad Aut\'onoma de Madrid}, \orgaddress{\street{Ciudad Universitaria de Cantoblanco}, \city{Madrid}, \postcode{28049}, \country{Spain}}}
	
	\affil*[2]{\orgdiv{Department of Mathematics}, \orgname{University of the Basque Country UPV/EHU}, \orgaddress{\street{Barrio Sarriena S/N}, \city{Leioa}, \postcode{48940}, \country{Spain}}}
		
	
	\abstract{In this paper, we give a rigorous proof for the expression of the angle between adjacent sides in the skew polygons appearing at rational times in the evolution of regular polygons of $M$ sides under the vortex filament equation. 
	The proof depends on showing that some exponential sums with arithmetic content are purely imaginary.}
	
	\keywords{Vortex filament equation, nonlinear Schr\"odinger equation, rotation matrices, trigonometric sums}
		
	\pacs[MSC]{35Q35, 35Q41, 11L03, 11R52}
	
	\maketitle

\section{Introduction}

The binormal flow,
\begin{equation*}
\Xt = \kappa\mathbf b,
\end{equation*}
where $t$ is the time, $\kappa$ is the curvature, and $\mathbf b$ is the binormal component of the Frenet-Serret formulas, appeared for the first time in 1906 in \cite{darios} as an approximation of the dynamics of a vortex filament under the Euler equations, and is equivalent to
\begin{equation}
	\label{e:xt}\Xt = \Xs\wedge\Xss,
\end{equation}
where $\wedge$ is the usual cross-product, and $s$ is the arc-length parameter. The length of the tangent vector $\T = \X_s$ remains constant, so we can assume, without loss of generality, that $\|\T\| = 1$, where $\|\cdot\|$ denotes the Euclidean norm. Differentiating \eqref{e:xt} with respect to $s$, we get the Schr\"odinger map onto the sphere:
\begin{equation}
	\label{e:schmap}\Tt = \T\wedge\Tss.
\end{equation}
The evolution of \eqref{e:xt}-\eqref{e:schmap} for $\X(s, 0)$ having corners has received considerable attention in the last years, and especially relevant is the case when the image of $\X(s, 0)$ is a planar regular polygon $\mathcal{P}$ of $M$ sides, which we consider here. In what follows, we recall those ideas that are needed in this paper; for a complete description of the problem, we refer to \cite{HozVega2014,HozVega2014b,HozVega2018}.

Since \eqref{e:xt} and \eqref{e:schmap} are invariant with respect to rotations, we can suppose that  $\X(s, 0)$ and $\T(s, 0)$ live in the plane $XY$, which we identify with $\mathbb C$. 
As we have assumed that $\X(s, 0)$ is parameterized by arc length, imposing that $\mathcal{P}$ has length $2\pi$, we can consider 
\begin{equation}\label{e:init}
 \X(\cdot, 0):[0,2\pi)\longrightarrow\mathcal{P}
 \qquad\text{linear on }I_k=\Big(\frac{2\pi k}{M},\frac{2\pi (k+1)}{M}\Big),\quad 0\le k<M,
\end{equation}
as a $2\pi$-periodically extendable function such that the image of  
of $I_k$ is the $k$-side of $\mathcal{P}$. After a rotation, we can also assume that this side is parallel to  $e^{2\pi ik/M}$, equivalently, 
$\T(s,0)=e^{2\pi ik/M}$, when $s\in I_k$. 
\medskip

The main result in 
\cite{HozVega2014}
states that, when time evolves, for $t=t_{p/q}=\frac{2\pi p}{qM^2}$
with $p/q$ an irreducible fraction (with $q>0$), 
$\X(\cdot, t_{p/q})$ describes a skew (in general, nonplanar) polygon of $Mq$ sides, if $q$ is odd, and of $Mq/2$ sides, if $q$ is even. This polygon has sides of the same length and all the angles between adjacent sides are equal. In
\cite{HozVega2014}, a conjecture is posed on the magnitude of these angles.
We prove here this conjecture with an argument involving a variation of the Gaussian sums. 

\

\begin{theorem}\label{t:main}
Let $\rho$ be the angle between two adjacent sides of the skew polygon given by $\X(s, t_{p/q})$, where $s$ varies. Then, we have
\begin{equation}
	\label{e:cosr2}
	\cos^q\left(\frac\rho2\right) =
	\cos\left(\frac{\pi}{M}\right)\text{ if $q$ is odd},
	\quad\text{and}\quad
	\cos^q\left(\frac\rho2\right) =
	\cos^{2}\left(\frac{\pi}{M}\right)\text{ if $q$ is even}.
\end{equation}
\end{theorem}

{\it Remark.}
To avoid any confusion, we want to make clear that the uniqueness of the nonlinear equation \eqref{e:xt} under the nonsmooth initial condition \eqref{e:init} is unclear (see \cite{BanicaVega2024a} for some recent developments in this regard). Here, we are considering the tentative distributional solution introduced in \cite{HozVega2014}. There, an expression for $\X(s, t_{p/q})$ is deduced as a realization of the Talbot effect and the Galilean invariance of the nonlinear Schr\"odinger equation. There is not an analogue for the irrational time case, in which some fractal like structures show up (note the resemblance to the linear setting in \cite{BeKl}).
\medskip

To illustrate the shape of the polygons we are interested in, we have taken $M = 5$ and considered the evolution of a regular planar pentagon. In this case, $\T$ is periodic in time with period $2\pi/M^2 = 2\pi/25$, so we have plotted, on the left-hand side of Figure~\ref{f:pentagon}, $\X(s, t)$ at four times: $t = 0$ (initial datum), $t = t_{1/3} = 2\pi/75$ (one third of the time period), $t = t_{2/3} = 4\pi/75$ (two thirds of the time period) and $t = t_1 = 2\pi/25$ (one time period). Since the center of mass moves upward with constant speed, we recover at $t = 2\pi/25$ the regular planar pentagon, but lifted. On the other hand, $t = 2\pi/75$ and $t = 4\pi/75$ correspond to $q = 3$, so we have exactly $Mq = 15$ sides at those times; this can also be appreciated on the right-hand side of Figure~\ref{f:pentagon}, where we have plotted the exact values of the components of $\T(s, t)$ at time $t_{1/3} = 2\pi/75$, which are piecewise discontinuous, and consist each one of $15$ segments. Let also mention that the angle between any two adjacent sides of the $15$-sided skew polygons at $t = 2\pi/75$ and $t = 4\pi/75$ is, from \eqref{e:cosr2}, $\rho = 2\arccos(\cos^{1/3}(\pi/5)) = 0.74295\ldots$
\begin{figure}[!htbp]
	\label{f:pentagon}
	\includegraphics[width=0.5\textwidth,height=0.4\textwidth]{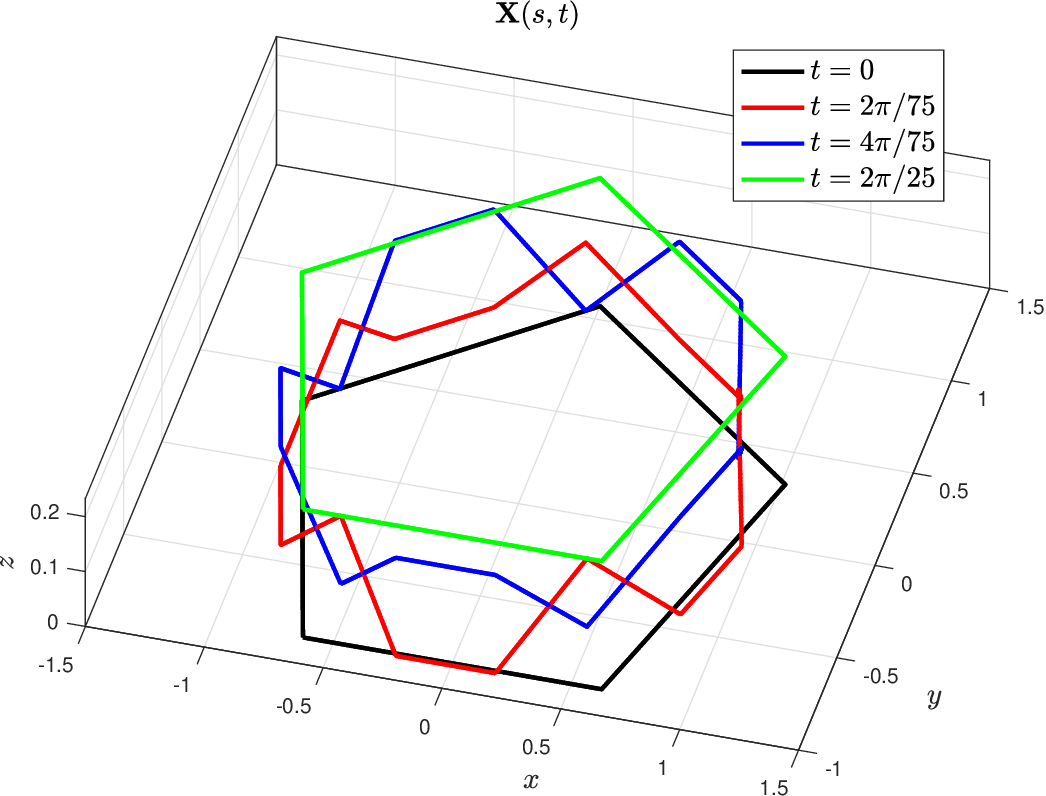}\includegraphics[width=0.5\textwidth,height=0.4\textwidth]{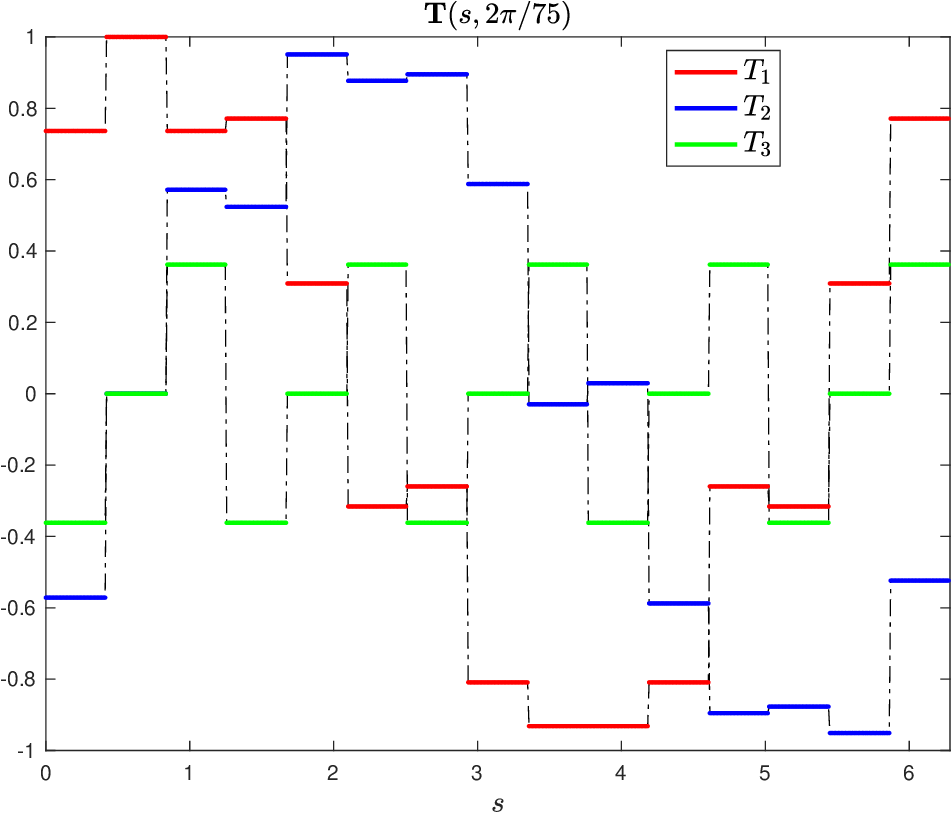}
	\caption{Evolution of a pentagon. Left: $\X(s, t)$ at $t = 0$, $t = t_{1/3} = 2\pi/75$, $t = t_{2/3} = 4\pi/75$ and $t = t_1 = 2\pi/25$. Right: $\T(s, t)$ at $t = t_{1/3} = 2\pi/25$.}
\end{figure}

With the tools explained in \cite{HozVega2014}, Theorem~\ref{t:main} can be translated into Theorem~\ref{t:rotation} below, which has a geometric statement and also an arithmetic flavor relying on the Gauss sums. 

Before stating this geometric result, we need to introduce some notation. 

The expression $R(\vec{v},\alpha)$ denotes the rotation in $\mathbb{R}^3$ of axis $\vec{v}$, assumed to be a unit vector, and angle $\alpha$.  As we have mentioned above, we identify complex numbers and vectors in $\{z=0\}$; then, $R\big(e^{i\theta},\alpha\big)$
is an abbreviation for 
$R\big((\cos\theta, \sin\theta,0),\alpha\big)$.
We employ the usual notation for divisibility in number theory: $a\mid b$ means $b/a\in\Z$, and $a\nmid b$ indicates the absence of divisibility.
Finally, $G(-p,n,q)$, with $p$ and $q$ coprime as before and $n\in\mathbb{Z}$, denotes the generalized Gauss sum 
\[
 G(-p,n,q)
 :=
 \sum_{k=0}^{q-1}
 \exp\Big(
 \frac{-pk^2+nk}{q}
 \Big).
\]

Now we are ready to state the geometric result from which Theorem~\ref{t:main} is deduced.

\begin{theorem}\label{t:rotation}
 Let $\theta_n$ be the argument of the generalized Gauss sum $G(-p,n,q)$, and define the rotation $\mathcal R$ as 
 \[
\mathcal R = 
\prod_{\substack{n=0 \\ 4\nmid q-2n}}^{q-1}
R\big(e^{i\theta_{q-1-n}},\rho\big).
\]
 Then, $\mathcal R$ is a rotation of angle $2\pi/M$, if and only if \eqref{e:cosr2} holds. 
\end{theorem}

{\it Remarks.}
The condition $4\nmid q-2n$ avoids exactly the cases in which the argument is undefined, because the Gauss sums vanish (see for instance  \cite{ChSa} for an elementary proof).
Note that, for $q$ odd, this condition is always fulfilled. In this noncommutative situation, it is important to respect the standard ordering of the product in the statement. For instance, for $q$ odd, it is 
$
R\big(e^{i\theta_{q-1}},\rho\big)
R\big(e^{i\theta_{q-2}},\rho\big)
\cdots
R\big(e^{i\theta_{0}},\rho\big)
$.

\medskip

Our approach to prove Theorem~\ref{t:rotation} is as follows. In \S\ref{s:trig}, the result is deduced assuming that a certain trigonometric sum involving $\theta_n$ vanishes; this only requires basic manipulations involving Pauli matrices. In 
\S\ref{s:quad}, using modular arithmetic and the quadratic nature of the Gauss sums, the trigonometric sum is related to the real part of a quadratic exponential sum. Finally, in 
\S\ref{s:vrp}, with more involved arithmetic arguments, it is shown that the quadratic exponential sum is purely imaginary.

\section{The trigonometric sum formulation}\label{s:trig}

Given $k,N\in\Z^+$, consider the set 
\[
 \mathcal{I}_k^N
 =
 \big\{
 \vec{n}=(n_1,\dots, n_k)\in\mathbb{Z}^k
 \,:\,
 0\le n_1<n_2<\cdots<n_k<N
 \big\}.
\]
With the same notation as in Theorem~\ref{t:rotation} we introduce the trigonometric sum 
\[
 \mathcal{T}_k
 =
\sum_{\substack{\vec{n}\in \mathcal{I}_{2k}^q \\ 4\nmid 2n_j+2-q}}
   \cos\big(\theta_{n_1}-\theta_{n_2}+\cdots -\theta_{n_{2k}}\big),
\]
where the divisibility condition $4\nmid 2n_j+2-q$ reflects that of Theorem~\ref{t:rotation} after the map $n\mapsto q-1-n$, which changes $\theta_{q-1-n}$ into $\theta_n$.   
Again, note that this condition only imposes an actual restriction for $q$ even.
In the next sections, we will show the following proposition.

\medskip

\begin{proposition}\label{p:Tk0}
 For any $0<2k\le q$, the sum $\mathcal{T}_k$ vanishes.
\end{proposition}

\

It is well known that there is a double cover $\text{SU}(2)\longrightarrow \text{SO}(3)$, in such a way that 
\[
 S(\vec{n},\alpha)
 :=
 \cos \Big(\frac{\alpha}{2}\Big)\, I
 +\vec{n}\cdot\vec{\sigma}
 \sin \Big(\frac{\alpha}{2}\Big)
 \longmapsto 
 R(\vec{n},\alpha),
\]
where 
$\vec{n}\cdot\vec{\sigma}=n_1\sigma_1+n_2\sigma_2+n_3\sigma_3$ and $\sigma_1$, $\sigma_2$, $\sigma_3$ are the standard Pauli matrices. This means that 
\begin{equation}\label{spinors}
 R(\vec{n}_1,\alpha_1)
 R(\vec{n}_2,\alpha_2)
 =
 R(\vec{n}_3,\alpha_3)
 \quad\Longleftrightarrow\quad
 S(\vec{n}_1,\alpha_1)
 S(\vec{n}_2,\alpha_2)
 =
 \pm S(\vec{n}_3,\alpha_3).
\end{equation}
The uncertainty in the sign is related to topological topics exemplified by Dirac's belt trick. 
Usually \eqref{spinors} is proved in the context of Lie group theory but, in fact, it can be deduced using basic linear algebra \cite{chamizo}.

To prove Theorem~\ref{t:rotation} assuming Proposition~\ref{p:Tk0}, apart from \eqref{spinors}, the only fact involving Pauli matrices that we are going to use is the simple identity stated in the following auxiliary result. 
\medskip

\begin{lemma}\label{l:trace}
 Let $\varphi_0,\,\varphi_1,\,\dots, \,\varphi_{N-1}$ be the angles in the $XY$ frame of  the unit vectors
 $\vec{v}_0,\,\vec{v}_1,\,\dots, \,\vec{v}_{N-1}\in\mathbb{R}^2\times\{0\}$. Then, we have the polynomial identity
 \[
  \frac 12\,
  \text{\rm Tr}
  \prod_{n=0}^{N-1}
  \big(x I +i\vec{v}_n\cdot\vec{\sigma}\big)
  =
  \sum_{k=0}^{\lfloor N/2\rfloor}
  x^{N-2k}
  \sum_{\vec{n}\in\mathcal{I}_{2k}^N}
  \cos\big(\varphi_{n_1}-\varphi_{n_2}+\cdots -\varphi_{n_{2k}}\big),
 \]
 where $\text{\rm Tr}$ indicates the trace, and for $k=0$, the inner empty set is defined as~$1$ by convention. 
\end{lemma}

\begin{proof}
 Expanding the product,
 \[
  \frac 12\,
  \text{\rm Tr}
  \prod_{n=0}^{N-1}
  \big(x I +i\vec{v}_n\cdot\vec{\sigma}\big)
  =
  \sum_{k=0}^{N}
  x^{N-k}
  \sum_{\vec{n}\in\mathcal{I}_{k}^N}
  \frac 12\,
  \text{\rm Tr}\big(
  (\vec{v}_{n_1}\cdot\vec{\sigma})\cdots (\vec{v}_{n_k}\cdot\vec{\sigma})
  \big),
 \]
 where we still use the convention about the empty sum. For $\vec{a}=(\cos\varphi,\sin\varphi,0)$ and 
 $\vec{b}=(\cos\psi,\sin\psi,0)$, we have 
 \[
  (\vec{a}\cdot\vec{\sigma})
  (\vec{b}\cdot\vec{\sigma})
  =
  \begin{pmatrix}
   0& e^{-i\varphi}
   \\
   e^{i\varphi}& 0
  \end{pmatrix}
  \begin{pmatrix}
   0& e^{-i\psi}
   \\
   e^{i\psi}& 0
  \end{pmatrix}
  =
  D\big(e^{i(\varphi-\psi)}\big),
 \]
 where $D(z)$ means the diagonal matrix with $\overline{d}_{11}=d_{22}=z$. 
 Applying this with $\vec{a}=\vec{v}_{n_{2r-1}}$ and $\vec{b}=\vec{v}_{n_{2r}}$, for $1\le r\le k/2$, we conclude that 
 \[
  (\vec{v}_{n_1}\cdot\vec{\sigma})\cdots (\vec{v}_{n_k}\cdot\vec{\sigma})
  =
  D\Big(\exp\big(i\sum_{r=1}^k(-1)^r\varphi_{n_r}\big)\Big),
 \]
 for $k$ even, and hence, half of its trace is the cosine in the statement. On the other hand, it is traceless for $k$ odd, because it is the product of $D\big(\exp\big(i\sum_{r=1}^{k-1}(-1)^r\varphi_{n_r}\big)\big)$ by $\vec{v}_{n_k}\cdot\vec{\sigma}$ that has zero diagonal. 
 Renaming $k$ as $2k$, the proof is complete.
\end{proof}

Now, we are ready to prove Theorem~\ref{t:rotation} assuming Proposition~\ref{p:Tk0}.

\begin{proof}[Proof of Theorem~\ref{t:rotation}]
 Let 
 $\vec{v}_n=\big(\cos\theta_{q-1-n}, \sin\theta_{q-1-n}, 0\big)$. Thanks to \eqref{spinors} and recalling that $e^{i\theta_{q-1-n}}$ represents the vector $\vec{v}_{n}$, the formula in the statement is equivalent to
 \[
  \prod_{\substack{n=0 \\ 4\nmid q-2n}}^{q-1}
  S\big(\vec{v}_n,\rho\big)
  =
  \pm S(\vec{v}, 2\pi/M),
 \]
 for some unit vector $\vec{v}\in\mathbb{R}^3$.
 
 Note that the matrices $S(\vec{v}, \alpha)$ with $\alpha=2\pi/M$ are characterized by the property $\frac 12 \,\text{\rm Tr}\,S(\vec{v}, \alpha)=\cos\frac{\pi}{M}$, because the entries in its diagonal are $\cos\frac{\alpha}{2}\pm i v_3\sin\frac{\alpha}{2}$. Hence, we have to prove
 \[
  \Bigg|
  \frac 12\, \text{\rm Tr}\, 
  \prod_{\substack{n=0 \\ 4\nmid q-2n}}^{q-1}
  \Big(
  \cos
  \frac{\rho}{2}
  +i\vec{v}_n\cdot\vec{\sigma}
  \sin
  \frac{\rho}{2}
  \Big)
  \Bigg|
  =
  \cos\frac{\pi}{M},
 \]
 if and only if \eqref{e:cosr2} holds. 
 Take $N=q$, $\varphi_n=\theta_{q-1-n}$ and $x=\cot\frac{\rho}{2}$
 in Lemma~\ref{l:trace}. By Proposition~\ref{p:Tk0}, the sum vanishes except for $k=0$, and the previous equality for $q$ odd  reads $\cos^q\frac{\rho}{2}=\cos\frac{\pi}{M}$, which is \eqref{e:cosr2}.
 Note that $4\nmid q-2n_j$ translates into $4\nmid 2n_j+2-q$ under $n_j\mapsto q-1-n_j$ (the map which passes from $\varphi_n$ in Lemma~\ref{l:trace} to $\theta_n$ in $\mathcal{T}_k$), and it still does not impose an actual condition for $q$ odd. 
 
 The same argument works for $q$ even restricting the set of vectors $\vec{v}_n$ to those with $n$ satisfying the congruence condition. Since there are $q/2$ of them (the number of $n$'s with $4\nmid q-2n$), the result is $\cos^{q/2}\frac{\rho}{2}=\cos\frac{\pi}{M}$, the second part of \eqref{e:cosr2}.
\end{proof}

\section{A quadratic exponential sum}\label{s:quad}

One could find formulas for $\theta_n$ using the evaluation of the generalized Gauss sums, but it leads to quite a number of cases and arithmetic subtleties related to quadratic residues. We avoid these complications thanks to the following result that is enough to deduce that the phases in $\mathcal{T}_k$ can be substituted by a quadratic polynomial. The proof essentially reduces to completing squares in the generalized Gauss sums.
\medskip

\begin{lemma}\label{l:th2quad}
 Let $\theta_n$ be, as before, the argument of $G(-p,n,q)$, and 
 $\delta = \frac{1-(-1)^q}{2}\in \{0,1\}$ the parity of $q$. For each $p$ and $q$ fixed, there exists $a\in\Z$ coprime with $q$, and $b\in\mathbb{R}$, such that
 \[
  \frac{2\pi a}{q}
  \Big(\frac{n}{2-\delta}\Big)^2
  +b-\theta_n
 \]
 is an integer multiple of $2\pi$, for any $0\le n<q$ with $4\nmid 2n+2-q$.
\end{lemma}

\begin{proof}
 Along this proof, we adopt the common notation in elementary number theory consisting in indicating the inverse modulo $q$ with a bar; for instance, $\overline{p}p\equiv 1\pmod{q}$.
 Completing squares, if $q$ is odd, for $k\in\mathbb{Z}$,
 \[
  -pk^2+nk
  \equiv
  -p(k-\overline{2p}n)^2 +\overline{4p}n^2
  \pmod{q}.
 \]
 If $q$ is even, $\overline{2}$ does not exist. We use instead the following formula, with 
 $\epsilon = \frac{1-(-1)^n}{2}\in\{0,1\}$ being the parity of $n$:
 \[
  -pk^2+nk
  \equiv
  -p\Big(k-\overline{p}\frac{n-\epsilon}{2}\Big)^2 
  +\epsilon\Big(k-\overline{p}\frac{n-\epsilon}{2}\Big)
  -\epsilon\frac{\overline{p}}{4}
  +\frac{\overline{p}}{4}n^2
  \pmod{q}.
 \]
 Note that $4\nmid 2n+2-q$ is equivalent to $2\mid n-q/2$ for $q$ even; hence, $\epsilon$ is constant for all valid values of $n$. In fact, $1-\epsilon$ is the parity of $q/2$.
 
 Choosing $a=\overline{4p}$ for $q$ odd and $a=\overline{p}$ for $q$ even, the previous formulas show that $G(-p,n,q)\exp\big({-2}\pi i\frac aq \big(\frac{n}{2-\delta}\big)^2\big)$
 equals
 $G(-p,0,q)$, for $q$ odd
 and it equals
 $G(-p,\epsilon,q)\exp\big({-\pi} i\frac{\epsilon a}{2q}\big)$ for $q$ even.
 These quantities do not depend on $n$, and the result follows after taking $b$ as their argument. 
\end{proof}

Consider now the quadratic exponential sum
\[
 \mathcal{E}_k
 =
\sum_{\substack{\vec{n}\in \mathcal{I}_{2k}^q \\ 4\nmid 2n_j+2-q}}
   \exp
   \Big(
   \frac{2\pi i a}{(2-\delta)^2 q}
   Q_k(\vec{n})
   \Big),
   \qquad\text{with}\quad
   Q_k(\vec{n})=n_1^2-n_2^2+\dots -n_{2k}^2.
\]
An immediate consequence of Lemma~\ref{l:th2quad} is given by the following corollary.
\medskip

\begin{corollary}\label{c:realE}
 For any $0<2k\le q$, we have $\mathcal{T}_k=\Re \mathcal{E}_k$.
\end{corollary}

\section{The vanishing of the real part}\label{s:vrp}

Bearing in mind Corollary~\ref{c:realE} and our considerations in \S\ref{s:trig}, the main result follows via Proposition~\ref{p:Tk0}, thanks to the following proposition.
\medskip

\begin{proposition}
 The real part of $\mathcal{E}_k$ is zero, for any $0<2k\le q$ and any $a$ coprime with $q$.
\end{proposition}
\medskip 

For the proof, it will be important to introduce a bijection $f_h:\II_k^N\longrightarrow\II_k^N$.
Given $h\in\Z$ and $\vec{n}\in\II_k^N$, the vector $\vec{n}+h\vec{1}$ modulo $N\in\N$, with $\vec{1}=(1,1,\dots, 1)\in\Z^k$, can be transformed into an element of $\II_k^N$ by a circular permutation of its coordinates; 
we take this as the definition of $f_h(\vec{n})$. 
For instance, if $N = 11$, $\vec{n} = (3, 5, 8, 9)\in\II_4^{11}$ and $h = 4$, then $\vec{n} + h\vec{1} = (7, 9, 12, 13)$, which is $(7, 9, 1, 2)$ modulo $11$; then, performing a circular permutation, it becomes $(1, 2, 7, 9)\in\II_4^{11}$. Therefore, in this example, $f_4(3, 5, 8, 9) = (1, 2, 7, 9)$. Note that $f_h$ is a bijection on $\II_k^N$, because it can be inverted by taking $\vec{n}-h\vec{1}$, i.e., $f_h^{-1}\equiv f_{-h}$.

\begin{proof}
 We distinguish the cases with odd and even $q$. Both follow the same scheme, but there is a slight difference due to a $2$ factor that could ruin a divisibility condition. It is important to keep in mind the elementary formula for $q\in\N$ and $c\in\Z$:
 \begin{equation}\label{e:runity}
  \sum_{h=0}^{q-1}
  e^{2\pi i c/q}=0,
  \qquad\text{if and only if}\quad q\nmid c.
 \end{equation}

 Consider $q$ odd. As circular permutations of the variables of the quadratic form $Q_k$ only may affect to its sign, we have that $Q_k\big(f_h(\vec{n})\big)$ equals $\pm Q(\vec{n}+h\vec{1})$ modulo $q$. Hence, the real part of $\exp\big(2\pi i aQ_k(f_h(\vec{n}))/q\big)$ remains invariant when $f_h(\vec{n})$ is replaced by $\vec{n}+h\vec{1}$, and we have, averaging over $h$,
 \begin{equation}\label{e:expodd}
  q\Re \mathcal{E}_k
  =
  \sum_{\vec{n}\in\mathcal{I}_{2k}^q}
  \Re
  \sum_{h=0}^{q-1}
  \exp\Big(\frac {2\pi ia }q Q(\vec{n}+h\vec{1})\Big). 
 \end{equation}
 Expanding 
 $Q_k(\vec{n}+h\vec{1})$, we get $Q_k(\vec{n})+2hL_k(\vec{n})$,
 with 
 $L_k(\vec{x})=\sum_{k=1}^{2d}(-1)^kx_k$.
 For each $\vec{n}\in\mathcal{I}_{2k}^q$, grouping together consecutive coordinates, we have $0<L_k(\vec{n})<q$, in particular $q\nmid L(\vec{n})$, and the result is deduced from \eqref{e:runity}.
 
 For $q$ even, we are going to apply a variant of the previous argument modulo $q'=q/2$. 
 If $\epsilon\in\{0,1\}$ is the parity of $q'$, the condition $4\nmid 2n_j+2-q$ implies that $\epsilon$ is also the parity of each $n_j$, and a change of variables $n_j\mapsto 2n_j+\epsilon$ allows to write the exponential sum as 
 \[
  \mathcal{E}_k
  =
  \sum_{\vec{n}\in \smash{\mathcal{I}_{2k}^{q'}}}
   \exp
   \Big(
   \frac{\pi i a}{2 q}
   Q_k(2\vec{n}+\epsilon \vec{1})
   \Big)
  =
  \sum_{\vec{n}\in \smash{\mathcal{I}_{2k}^{q'}}}
   \exp
   \Big(
   \frac{2\pi i a}{q}
   \big(Q_k(\vec{n})+\epsilon L_k(\vec{n})\big)
   \Big).
 \]
 A crucial point is that, for $x$ modulo $q'$, the value of $P(x)=x^2+\epsilon x$ is well defined modulo $q$. In other words, $P(x+kq')-P(x)$ is divisible by $q$. Then, using the bijection $f_h$, as before, but this time modulo $q'$, after averaging over   $h$ as in \eqref{e:expodd}, it follows that
 \[
  q'\Re \mathcal{E}_k
  =
  \sum_{\vec{n}\in \smash{\mathcal{I}_{2k}^{q'}}}
  \Re
  \sum_{h=0}^{q'-1}
   \exp
   \Big(
   \frac{2\pi i a}{q}
   \big(Q_k(\vec{n})+\epsilon L_k(\vec{n})+ 2hL_k(\vec{n})\big)
   \Big).
 \]
 A factor $\exp\big(2\pi i ahL_k(\vec{n})/q'\big)$ can be extracted from the last expression and, again, \eqref{e:runity} gives the result.
\end{proof}

\section*{Acknowledgments}
Fernando Chamizo is partially supported 
by the PID2020-113350GB-I00 grant of the MICIU (Spain) and by ``Severo Ochoa Programme for Centres of Excellence in R{\&}D'' (CEX2019-000904-S).
Francisco de la Hoz is partially supported by the research group grant IT1615-22 funded by the Basque Government, and by the project PID2021-126813NB-I00 funded by MICIU/AEI/10.13039/501100011033 and by ``ERDF A way of making Europe''.

%\bibliography{VFEquaternions}


\end{document}